\documentclass[11pt]{article}
\usepackage{graphicx}
\usepackage{amssymb}
\usepackage{epstopdf}
\usepackage{amsmath,amsthm,amscd,amssymb}
\usepackage{latexsym}
\usepackage[colorlinks,citecolor=red,pagebackref,hypertexnames=false]{hyperref}
\usepackage{geometry}                
\numberwithin{equation}{section}

\theoremstyle{plain}
\newtheorem{theorem}{Theorem}[section]
\newtheorem{lemma}[theorem]{Lemma}

\theoremstyle{definition}

\newtheorem{case[theorem]}{Case}

\theoremstyle{remark}
\newtheorem{remark}[theorem]{Remark}

\numberwithin{equation}{section}

\begin{document}

\title{Erd\H{o}s Type Problems in Modules over Cyclic Rings}

\author{Esen Aksoy Yazici}

\maketitle

\begin{abstract} In the present paper, we study various Erd\H{o}s type geometric problems in the setting of the integers modulo $q$, where $q=p^l$ is an odd prime power. More precisely, we prove certain results about the distribution of triangles and triangle areas among the points of $E\subset \mathbb{Z}_q^2$. We also prove a dot product result for $d$-fold product subsets $E=A\times \ldots \times A$ of $\mathbb{Z}_q^d$, where $A\subset \mathbb{Z}_q$. \end{abstract}

\section{Introduction}

\vskip.125in

 The classical Erd\H{o}s distance problem asks for the number of distinct distances determined by $n$ points in $\mathbb{R}^d$. In \cite{E46},  Erd\H{o}s conjectured that the minimum number of distinct distances determined by $n$ points  in the Euclidean plane is $C\frac{n}{\sqrt{ log n}}$. Several results have been given in this direction and recently L. Guth and G. Katz \cite{GK11} settled the conjecture, up to a square root $log$ factor showing that $n$ points determine at least $C\frac{n}{log n}$ distances. We shall note here that in higher dimensions the distance problem is still open with the best known results are due to Solymosi and Vu, \cite{SV04}. 

A natural generalization of the distance problem is the distribution of $k$-simplices. In this setting, for $k=2$, two triangles are said to be congruent if their corresponding side lengths are equal and we will denote by $T_2(E)$ the set of congruence classes of triangles. In \cite{GILP}, for subsets $E$ of $\mathbb{R}^2$, Greenleaf, Iosevich, Liu and Palsson proved that if $dim_H(E)>\frac{8}{5}$, then $\mathcal{L}^3(T_2(E))>0.$

A variant of the distance problem is the dot product-volume set problem.  Given $n$ points in the plane one can easily see that the number of distinct triangle areas determined by the points can be as large as $n \choose 3$ if the points are in general position. In 2008, \cite{P08}, Pinchasi settled a long standing conjecture of Erd\H{o}s, Purdy and Straus proving that the number of distinct areas of triangles determined by a non-collinear point set of size $n$ is at least $\lfloor{\frac{n-1}{2}}\rfloor$.

One can ask similar discrete questions in the context of vector spaces over finite fields $\mathbb{F}_q^d$, or modules over finite cyclic rings $\mathbb{Z}_q^d$. Several problems have been studied by various authors in the context of finite fields, see for example \cite{BIP}, \cite{BHIPR13}, \cite{CEHIK12}, \cite{CHISU},\cite{HI}, \cite{HIKR11}, \cite{IRZ12} and the references therein. Indeed, the techniques and results for the problems in the context of finite fields are an analogous version of those in Euclidean space. 

For $G=\mathbb{F}_q$ or $\mathbb{Z}_{q}$, and $x,y \in G^d$, we can consider the following distance map

$$ \lambda: (x,y)\longmapsto\|x-y\|=(x_1-y_1)^2+\ldots+(x_d-y_d)^2.$$

This map does not induce a metric on $G^d$. However, we will rely on the fact that it is invariant under orthogonal transformations in the subsequent sections. The Erd\H{o}s-Falconer distance problem in $G^d$ asks for a threshold on the size  $E\subset G^d$ so that the distance set of $E$,
$$\Delta(E):=\{\|x-y\|: x,y\in E\},$$
is about the size $q$. 

In \cite{IR07},  Iosevich and Rudnev prove that for $E\subset \mathbb{F}_q^d$ if $|E|>Cq^{\frac{d+1}{2}}$ for a sufficiently large constant $C$, then $\Delta(E)=\mathbb{F}_q$. Note that this result is in parallel to the Falconer result for the subsets of $\mathbb{R}^d$. 

Given a subset $E$ of $G^d$, the dot product is defined as $$\prod(E)=\{x.y:\; x,y\in E \},$$ where dot product of two vectors is defined by the usual formula  $$x.y=x_1y_1+\ldots+ x_dy_d,$$ for $x=(x_1,\ldots, x_d)$ and $y=(y_1,\ldots, y_d)$.

We also define the $d$- dimensional non-zero volumes of  $(d+1)$-simplices whose vertices are in $E$ by   
$$V_d(E)=\{det(x^{1}-x^{d+1},\ldots, x^{d}-x^{d+1}): \;x^{j}\in E\}\setminus\{0\}.$$ The distribution of triangle areas among the points of a subset $E$ of $\mathbb{F}_q^2$ is studied by Iosevich, Rudnev, and Zhai in \cite{IRZ12}. The method uses a point-line incidence theory and the result is the following. If $E\subset \mathbb{F}_q^2$ with $|E|>q,$ then $|V_2(E)|\geq \frac{q-1}{2}$, and the triangles giving at least $\frac{q-1}{2}$ distinct areas can be chosen such that they share the same base.

In this paper, we turn our attention to the Erd\H{o}s-Falconer type problems in modules $\mathbb{Z}_q^d$ over the cyclic rings $\mathbb{Z}_q$, where $q=p^l$, $p$ is an odd prime, and prove the following results. Compared to configurations in vector spaces over finite fields, to overcome the difficulties arising  from the zero divisors in these cyclic rings, an extra arithmetical machinery is developed.


 
\vskip.125in 

\noindent \textbf{Acknowledgments.} The author is grateful to Alex Iosevich and Jonathan Pakianathan for their guidance and support.

\vskip.125in 
\subsection{Statement of main results}
\vskip.125in 
Let $$SO_2(\mathbb{Z}_q)=\{A\in M_2(\mathbb{Z}_q): AA^T=I,\;det(A)=1\}.$$ In this setting, two triangles 
$(x^1,x^2,x^3)$, $(y^1,y^2,y^3)$ in $\mathbb{Z}_q^2$ are said to be congruent
if $\exists$$\theta\in SO_2(\mathbb{Z}_q)$ such that
$$x^i-x^j=\theta(y^i-y^j) \; \text{for all}\; i, j.$$ 

Let $T_2(E)$ denote the set of congruence classes of triangles determined by the points of $E\subset\mathbb{Z}_q^2.$ We then prove the following.
\vskip.125in

\begin{theorem}\label{triconfig}
Suppose $E\subset \mathbb{Z}_q^2$ with $q=p^l$ and $p \equiv 3\; \text{mod}\; 4$. If $|E|\geq\sqrt[3]{3}p^{2l-\frac{1}{3}}$, then $|T_2(E)|\gtrsim q^3$.
\end{theorem}
In \cite{CIP}, Covert, Iosevich and Pakianathan prove an asymptotically sharp bound for the distance set  $\Delta(E)$ of $E\subset \mathbb{Z}_q^d$. More precisely, it is shown that if $E\subset \mathbb{Z}_q^d$, where $q=p^l$, and $|E|\gg l(l+1)q^{\frac{(2l-1)d}{2l}+\frac{1}{2l}}$, then $\Delta(E)\supset \mathbb{Z}_q^{*}$, where  $\mathbb{Z}_q^{*}$ denotes the set of unit elements of $\mathbb{Z}_q$. 
In Theorem \ref{triconfig} above, we study distribution of triangles for subsets $E$ of  $\mathbb{Z}_q^2$.

\vskip.125in

In Theorem \ref{Areainc}, we modify a method used in \cite{IRZ12} over finite fields, to prove a sufficient condition on the size of $E\subset \mathbb{Z}_q^2$ so that the number of distinct triangle areas determined by $E$ is about the size $q.$ 

For $E\subset \mathbb{Z}_q^2,$ let
$$V_2(E)=\{det(x^{1}-x^{3},x^{2}-x^{3}): \;x^{j}\in E\}\setminus\{0\}.$$

\begin{theorem}\label{Areainc}
Let $E\subset \mathbb{Z}_q^2$ where $q=p^l$. Suppose that $|E|>p^{2l-\frac{1}{2}}$. Then $|V_2(E)|\geq \frac{q}{4}\frac{1+p}{p}-1$. 
\end{theorem}

\vskip.125in

The following theorem can be seen as an improvement on the dot product result for an arbitrary subset $E$ of $\mathbb{Z}_q^d$ given in \cite{CIP}. For $E\subset \mathbb{Z}_q^d$, $q=p^l$, the authors in \cite{CIP} prove that if $|E|\gg lq^{\frac{(2l-1)d}{2l}+\frac{1}{2l}}$, then $\prod(E)\supset \mathbb{Z}_q^{*}.$ We take the set $E$ as a $d$-fold product of a subset $A$ of $\mathbb{Z}_q$, and remove the factor $l$ on the size of $E$ to get a positive proportion of the elements of $\mathbb{Z}_q$.

\vskip.125in

\begin{theorem}\label{dot E}
Let $E=\underbrace{A\times \ldots \times A}_\text{\text{d-fold}}$ be a subset $\mathbb{Z}_q^d$ , where $A\subset \mathbb{Z}_q$, $q=p^l$. Suppose $|E|\geq q^{d(\frac{2l-1}{2l})+\frac{1}{2l}}$. Then $\prod(E)\gtrsim q.$
\end{theorem}

\vskip.25in 

\subsection{Fourier Analysis in $\mathbb{Z}_q^d$}
 Let $ f,g : \mathbb {Z}_{q}^d\to \mathbb{C}$. The Fourier transform of $f$ is defined as 
$$\widehat {f}(m)=q^{-d}\sum_{x\in \mathbb {Z}_{q}^d}\chi(-x.m)f(x),$$
where $\chi(z)=exp(2\pi iz/q).$ 

We use the following properties:
\begin{equation*} 
  q^{-d}\sum_{x\in \mathbb {Z}_{q}^d}\chi(x.m) = \left\{
    \begin{array}{rl}
      1, & \text{if } m=0\\
      0, & otherwise
      \end{array} \right.  \qquad \text{(Orthogonality)}
\end{equation*}

\begin{equation*}
f(x)=\sum_{m\in \mathbb {Z}_{q}^d}\chi(x.m)\widehat{f}(m) \qquad \text{(Inversion)}
\end{equation*}
\begin{equation*}
\sum_{m\in \mathbb {Z}_{q}^d}|\widehat{f}(m)|^2=q^{-d}\sum_{x\in \mathbb {Z}_{q}^d}|f(x)|^2. \qquad \text{(Plancherel)}
\end{equation*}

\vskip.125in 

\section{Proof of Theorem \ref{triconfig}}
For the proof of Theorem \ref{triconfig} we will need the following lemmas.

Let us first denote by $SO_2({\mathbb{Z}_q})=\{A\in M_2(\mathbb{Z}_q): AA^{T}=I, \; det(A)=1\}$ the special orthogonal group.
\begin{lemma}\label{nonzero norm} Let $\xi=(\xi_1,\xi_2) \in \mathbb{Z}_q^2$, where $q=p^l$  and $p$ is an odd prime. If $\|\xi\|=\xi_1^2+\xi_2^2 \ne 0$, then $|Stab(\xi)|\leq p^{l-1}$, where $Stab$ is the stabilizer under the action of the special orthogonal group. 
\end{lemma}
\begin{proof}
Let $\xi=(x,y)\in \mathbb{Z}_q^2$. Since $\|\xi\|\ne0$, we can write $\|\xi\|=x^2+y^2=p^iu,\; 0\leq i\leq l-1,\; u\in \mathbb{Z}_q^*.$ Now if $A=\begin{bmatrix}a&b\\-b&a  \end{bmatrix}\in SO_2(\mathbb{Z}_q)$ fixes $\xi$, then from the identity 
\begin{equation*}
\begin{bmatrix}a&b\\-b&a \end{bmatrix}\begin{bmatrix}x\\y\end{bmatrix}=\begin{bmatrix}x\\y\end{bmatrix}
\end{equation*}
we get 
\begin{eqnarray*}
(a-1)x+by&=&0\\
-bx+(a-1)y&=&0,
\end{eqnarray*}
and hence
\begin{eqnarray}\label{a,b}
(a-1)(x^2+y^2)&\equiv 0&\;\text{mod}\;{p^l}, \nonumber \\
b(x^2+y^2)&\equiv 0&\; \text{mod}\; {p^l}. 
\end{eqnarray}
Putting $x^2+y^2=p^iu$ in (\ref{a,b}), we have
 \begin{equation}\label{a-b}
 a=p^{l-i}k+1,\; 0 \leq k<p^i,\; \text{and} \;\; b=p^{l-i}m,\; 0\leq m<p^i,
 \end{equation}
 where 
 \begin{equation}\label{detA}
 a^2+b^2\equiv 1\; \text{mod}\;{p^l}.
 \end{equation}

Now to conclude the argument, we claim that  for $a_0\ne a$ if $b_0$ and $b$ are satisfying $a_0^2+b_0^2\equiv 1\; \text{mod}\; {p^{l}}$ and $a^2+b^2\equiv 1\;  \text{mod}\; {p^{l}}$ and condition (\ref{a-b}), respectively, then $b_0\ne b$. This will prove the lemma, for then the number of pairs $(a,b)$ satisfying the conditions (\ref{a-b}) and (\ref{detA}) is at most the number of possibilities of $b$ which is  $p^i$.  This is at most  $p^{l-1}$ as the valuation of a nonzero element is at most $l-1$.

It remains to prove the claim and we will prove its contrapositive here. Suppose that $b_0=b$ and  $a_0^2+b_0^2\equiv 1\;\text{mod}\;{p^{l}}$,\; $a^2+b^2\equiv 1\;\text{mod}\;{p^{l}}$, so that $a_0^2\equiv a^2\;\text{mod}\;{p^l}$.  
Writing $a_0=p^{l-i}k_0+1$ and $a=p^{l-i}k+1$, It follows that 
\begin{eqnarray*}
(p^{l-i}k_0+1)^2&\equiv&(p^{l-i}k+1)^2\;\text{mod}\;{p^l},\\
p^{2l-2i}k_0^2+2p^{l-i}k_0&\equiv&p^{2l-2i}k^2+2p^{l-i}k\;\text{mod}\;{p^l},
\end{eqnarray*}
therefore,
\begin{eqnarray*}
p^{2l-2i}(k_0^2-k^2)+2p^{l-i}(k_0-k)&\equiv&0\; \text{mod}\;{p^l},\\
p^{l-i}(k_0-k)(p^{l-i}(k_0+k)+2)&\equiv&0\;\text{mod}\;{p^l}.
\end{eqnarray*}
Thus $p^{l} \mid p^{l-i}(k_0-k)(p^{l-i}(k_0+k)+2)$, and since $p\nmid p^{l-i}(k_0+k)+2$ as $p$ is odd, we must have $p^i\mid k_0-k<p^i$. Hence we have $k_0-k=0$, i.e., $k_0=k$ and therefore $a_0=a$.
\end{proof}

\begin{lemma}\label{zero norm}
 Let $\xi \in \mathbb{Z}_q^2\setminus(0,0)$, where $q=p^l$ and $p\equiv 3\;\text{mod}\;4$. If $\|\xi\|=0$, then $|Stab(\xi)|\leq p^{l-1}.$
\end{lemma}
For the proof of Lemma \ref{zero norm} we will use Hensel's Lemma.
\begin{lemma}[Hensel's Lemma]
Let $f(x)\in\mathbb{Z}[x]$, $f(r)\equiv 0\;\text{mod}\;p$ and $f'(r)\not\equiv0\;\text{mod}\;p$ so that r is a simple root of $f$ modulo $p$. Then for any $k\geq2$, there exists a unique $\hat{r}$ in $\mathbb{Z}_{p^k}$ such that $f(\hat{r})\equiv 0\;\text{mod}\;{p^k}$ with $\hat{r}\equiv r \;\text{mod}\;p$.
\end{lemma}

\begin{proof}[Proof of Lemma \ref{zero norm}]
We first note that as $p\equiv 3\;\text{mod}\;4$, for $\xi \in \mathbb{Z}_q^2\setminus(0,0)$, $\|\xi\|=0$ implies that $\xi=(p^mu, p^m v)$ for some $m\geq\frac{l}{2}$ and $u,v\in \mathbb{Z}_q^{*}$. Now if $A \in SO_2(\mathbb{Z}_q)$ fixes $\xi=(p^mu, p^m v)$, it can be readily shown that it also fixes $\eta=(-p^mv,p^mu)$. Hence $A$ also fixes $Span\{\xi, \eta\}=p^m(\mathbb{Z}_{p^l}\times\mathbb{Z}_{p^l})\cong \mathbb{Z}_{p^{l-m}}\times \mathbb{Z}_{p^{l-m}}.$

Since $\xi\ne(0,0)$, we have $m\leq l-1$. We shall note that  $\mathbb{Z}_{p^{l-m}}\times \mathbb{Z}_{p^{l-m}}$ is smallest, and hence the number of matrices $A$ that fixes  $\mathbb{Z}_{p^{l-m}}\times \mathbb{Z}_{p^{l-m}}$ is largest, when $m=l-1$. Therefore it is sufficient to consider the case $m=l-1$. In this case A fixes $p^{l-1}(\mathbb{Z}_{p^l}\times\mathbb{Z}_{p^l})\cong \mathbb{Z}_{p}\times \mathbb{Z}_{p}.$

We now write $$A=I_2+B,$$ where $I_2$ denotes the $2\times 2$ identity matrix and $B\in M_2(\mathbb{Z}_q)$. Then for any $y\in \mathbb{Z}_q^2$ we have 
$$Ap^{l-1}y=p^{l-1}y+Bp^{l-1}y,$$
so that $Bp^{l-1}y=0$ as $A$ fixes $p^{l-1}y$.
This implies that $B=pB'$ for some $B'\in M_2(\mathbb{Z}_q),$ and 
$$A=I_2+pB'\in \Gamma_{1}\cap SO_2(\mathbb{Z}_q).$$
 where $\Gamma_{1}$ denotes the matrices in $M_2(\mathbb{Z}_q)$ congruent to $I_2\;\text{mod}\; p$.

It follows that 
\begin{equation}\label{matrixset}
A\in \left\{ \begin{bmatrix}
a&-b\\
b&a
\end{bmatrix}
: a,b\in \mathbb{Z}_q,\;a^2+b^2\equiv 1\;\text{mod}\;q,\; a\equiv 1\;\text{mod}\;p,\;b\equiv 0\;\text{mod}\;p  \right\}.
\end{equation}
Now we count the number of matrices in (\ref{matrixset}). We first fix $b$. Since $b\equiv 0\;\text{mod}\;p$, we have $p^{l-1}$ choices for $b$.  Then we consider the polynomial $f(x)=x^2-(1-b^2)\in\mathbb{Z}[x]$. Note that $f(x)=x^2-1$ in $\mathbb{Z}_p$ as $b\equiv 0\;\text{mod}\;p$. Hence $1$ is a root of $f(x)$ and $f'(1)=2\ne 0$ in $\mathbb{Z}_p$ as $p$ is odd.  Hence by Hensel's Lemma there exists a unique $a$ in $\mathbb{Z}_q$ such that $f(a)=a^2-(1-b^2)=0$ in $\mathbb{Z}_q$ with $a \equiv 1\bmod p$. Therefore the number matrices of the form in (\ref{matrixset}) is $p^{l-1}$. This completes the proof.

\end{proof}

We make use of the following lemma from \cite{BHIPR13}.
\begin{lemma}\label{taylor} For any finite space $F$, any function $f : F\to \mathbb{R}_{\geq 0}$, and any $n\geq 2$ we have
$$\sum_{z\in F}f^n(z)\leq|F|\left(\frac{\|f\|_1}{|F|}\right)^{n}+\frac{n(n-1)}{2}\|f\|_{\infty}^{n-2}\sum_{z\in F}\left(f(z)-\frac{\|f\|_1}{|F|}\right)^2, $$

where $\|f\|_1=\sum_{z\in F}|f(z)|,$ and $\|f\|_{\infty}=max_{z\in F}f(z).$
\end{lemma}

Lastly, we state the following lemma from \cite{CIP} and use Remark \ref{SO size} for the proof of Theorem \ref{triconfig}.

\begin{lemma}\label{sphere} Let $d\geq2$ and $j\in \mathbb{Z}_q^*$, where $q$ is odd. Set $||x||=x_1^2+...+ x_d^2.$ Denote by $S_j=\{x\in \mathbb{Z}_q^d: ||x||=j\}$ the sphere of radius $j$. Then,
$$|S_j|=q^{d-1}(1+o(1)).$$
\end{lemma}
\begin{remark}\label{SO size} Note that 
\begin{equation}\label{matrixset2}
SO_2(\mathbb{Z}_q)= \left\{ \begin{bmatrix}
a&-b\\
b&a
\end{bmatrix}
: a^2+b^2=1\right\}
\end{equation}
and hence if we denote by $S_1$ the sphere of radius $1$ in $\mathbb{Z}_q^2$, then $|SO_2(\mathbb{Z}_q)|=|S_1| \sim q$, by Lemma \ref{sphere}.
\end{remark}

\begin{proof}[Proof of Theorem \ref{triconfig}]
 We first recall that $SO_2({\mathbb{Z}_q})=\{A\in M_2(\mathbb{Z}_q): AA^{T}=I, \; det(A)=1\}$
and define an equivalence relation on  $(\mathbb{Z}_q^2)^3$ as
$$(a,b,c)\sim (a',b',c')$$
 if $\exists \theta\in SO_2(\mathbb{Z}_q)$ with $a'=\theta a,\; b'=\theta b,\; c'=\theta c.$
 
For $E\subset \mathbb{Z}_q^2$ and $a,b,c \in \mathbb{Z}_q^2$, let 
\begin{equation*}
\mu (a,b,c)=|\{(x,y,z)\in E^3: \exists \theta\in SO_2(\mathbb{Z}_q)\; \text{such that} \; x-y=\theta a, y-z=\theta b, x-z=\theta c.\}|.
\end{equation*}

Note that $\mu(\theta a, \theta b, \theta c)=\mu(a,b,c)$ for all $\theta\in  SO_2({\mathbb{Z}_q})$, so $\mu$ can be viewed as a function $\mu: (\mathbb{Z}_q^2)^3/\sim \to \mathbb{Z}_{\geq 0}$.

Then by the Cauchy-Schwarz inequality, 
$$|E|^6=\left(\sum_{(a,b,c)\in(\mathbb{Z}_q^2)^3/\sim}\mu(a,b,c)\right)^2\leq |T_2(E)|\left (\sum_ {(a,b,c)\in(\mathbb{Z}_q^2)^3/\sim}\mu^2(a,b,c) \right),$$

where
\begin{eqnarray*}
|T_2(E)|=|\{(a,b,c)\in(\mathbb{Z}_q^2)^3/\sim: \mu(a,b,c)\neq 0\}|,
\end{eqnarray*}
which is equal to
$$|\{(a,b,c)\in(\mathbb{Z}_q^2)^3/\sim:\exists (x,y,z)\in E^3\: \text{and}\; \theta\in  SO_2({\mathbb{Z}_q})\; \text{such that}\;  x-y=\theta a, y-z=\theta b, x-z=\theta c\}|.$$

We have, 
\begin{eqnarray*}
\mu^{2}(a,b,c)&=&|\{(x,y,z,x',y',z')\in E^6 : \exists \theta_1, \theta_2\in SO_2(\mathbb{Z}_q)\; \text{such that}\\
&&x-y=\theta_1 a,\; x'-y'=\theta_2 a\\
&&y-z=\theta_1 b,\; y'-z'=\theta_2 b\\
&&x-z=\theta_1 c,\: x'-z'=\theta_2 c\}|\\
&=&|\{(x,y,z,x',y',z')\in E^6 : \exists \theta_1, \theta_2\in SO_2(\mathbb{Z}_q)\; \text{such that}\\
&&\theta_1^{-1}(x-y)=\theta_2^{-1}(x'-y')=a\\
&&\theta_1^{-1}(y-z)=\theta_2^{-1}(y'-z')=b\\
&&\theta_1^{-1}(x-z)=\theta_2^{-1}(x'-z')=c\}|
\end{eqnarray*}
so that
\begin{eqnarray*}
\sum_ {(a,b,c)\in(\mathbb{Z}_q^2)^3/\sim}\mu^2(a,b,c)&=&|\{(x,y,z,x',y',z')\in E^6 : \exists \theta_1, \theta_2\in SO_2(\mathbb{Z}_q)\; \text{such that}\\
&&\theta_1^{-1}(x-y)=\theta_2^{-1}(x'-y')\\
&&\theta_1^{-1}(y-z)=\theta_2^{-1}(y'-z')\\
&&\theta_1^{-1}(x-z)=\theta_2^{-1}(x'-z')\}|\\
&=& |\{(x,y,z,x',y',z')\in E^6 : \exists \theta \in SO_2(\mathbb{Z}_q)\; \text{such that}\\
&&\theta(x-y)=x'-y',\;\theta(y-z)=y'-z',\; \theta(x-z)=x'-z'\}|\\
&=& |\{(x,y,z,x',y',z')\in E^6 : \exists \theta \in SO_2(\mathbb{Z}_q)\; \text{such that}\\
&&x'-\theta x=y'-\theta y=z'-\theta z\}|.
\end{eqnarray*}
For a fixed $\theta  \in SO_2(\mathbb{Z}_q)$, let 
\begin{equation}\label{nu}
\nu_{\theta}(t)=|\{(u,v)\in E\times E:\; u-\theta(v)=t\}|.
\end{equation}
 Then we have 
$$ \nu_{\theta}^3(t) = |\{(x,y,z,x',y',z')\in E^6:x'-\theta x=y'-\theta y=z'-\theta z=t\}|, $$
and therefore
\begin{equation}\label{L2}
\sum_{(a,b,c)\in(\mathbb{Z}_q^2)^3/\sim}\mu^2(a,b,c)\leq \sum_{\substack{\theta\in  SO_2(\mathbb{Z}_q)\\t\in \mathbb{Z}_q^2}}\nu_{\theta}^3(t). 
\end{equation}
By Lemma \ref{taylor}, 
$$ \sum_{t\in \mathbb{Z}_q^{2}}\nu_{\theta}^3(t)\leq q^2\left(\frac{\|\nu_{\theta}\|_{1}}{q^2}\right)^{3}+3\|\nu_{\theta}\|_{\infty} \sum_{t\in \mathbb{Z}_q^{2}}\left(\nu_{\theta}(t)-\frac{\|\nu_{\theta}\|_{1}}{q^2} \right)^{2}$$
where
\noindent $\|\nu_{\theta}\|_{1}=\sum_{t\in \mathbb{Z}_q^{2}}\nu_{\theta}(t)=|E|^2$ and
\noindent $ \|\nu_{\theta}\|_{\infty}=sup_{t}|\nu_{\theta}(t)|\leq|E|$ as when we first fix $v$ in (\ref{nu}),  $u$ is uniquely determined.

It follows that
\begin{eqnarray*}
\sum_{t\in \mathbb{Z}_q^{2}}\nu_{\theta}^3(t)&\leq& q^{-4}|E|^6+3|E|\sum_{t\in \mathbb{Z}_q^{2}}\left(\nu_{\theta}(t)-\frac{\|\nu_{\theta}\|_{1}}{q^2} \right)^{2}\\
&\leq&q^{-4}|E|^6+3q^2|E|\sum_{ \xi\in \mathbb{Z}_q^{2}\setminus (0,0)}|\widehat{\nu}_{\theta}(\xi)|^2\;(\textit{by Plancherel Theorem})
\end{eqnarray*}
and thus
\begin{equation*}\sum_{\substack{\theta\in  SO_2(\mathbb{Z}_q)\\t\in \mathbb{Z}_q^2}}\nu_{\theta}^3(t)\leq|SO_2(\mathbb{Z}_q)|q^{-4}|E|^6+3q^2|E|\sum_{\substack{\theta\in  SO_2(\mathbb{Z}_q)\\ \xi\in \mathbb{Z}_q^{2}\setminus (0,0)}}|\widehat{\nu}_{\theta}(\xi)|^2
\end{equation*}
By Remark \ref{SO size}, $|SO_2(\mathbb{Z}_q)|\sim q$ and hence
\begin{equation}\label{cube}
\sum_{\substack{\theta\in  SO_2(\mathbb{Z}_q)\\t\in \mathbb{Z}_q^2}}\nu_{\theta}^3(t)\lesssim q^{-3}|E|^6+3q^2|E|\sum_{\substack{\theta\in  SO_2(\mathbb{Z}_q)\\ \xi\in \mathbb{Z}_q^{2}\setminus (0,0)}}|\widehat{\nu}_{\theta}(\xi)|^2
\end{equation}
Noting that
\begin{eqnarray*}
\nu_{\theta}(t)&=&\sum_{v\in \mathbb{Z}_q^2}E(v)E(t+\theta v)\\
&=&\sum_{v,\alpha \in \mathbb{Z}_q^2 }E(v)\chi (\alpha.(t+\theta v))\widehat{E}(\alpha)\\
&=&\sum_{v,\alpha \in \mathbb{Z}_q^2 }\widehat{E}(\alpha)\chi(t.\alpha)E(v)\chi(\alpha.\theta v)\\
&=&\sum_{\alpha \in \mathbb{Z}_q^2}\widehat{E}(\alpha)\chi(t.\alpha)\sum_{v\in  \mathbb{Z}_q^2 }\chi (\alpha.\theta v)E(v)\\
&=&\sum_{\alpha \in \mathbb{Z}_q^2}\widehat{E}(\alpha)\chi(t.\alpha)\sum_{v\in  \mathbb{Z}_q^2 }\chi(\theta^{T}(\alpha).v)E(v)\\
&=&q^2\sum_{\alpha  \in \mathbb{Z}_q^2}\widehat{E}(\alpha)\chi(t.\alpha)\widehat{E}(-\theta^{T}(\alpha)),
\end{eqnarray*}
and 
\begin{eqnarray*}
\widehat{\nu}_{\theta}(\xi)&=&q^{-2}\sum_{t\in \mathbb{Z}_q^2}\chi(-t.\xi)\nu_{\theta}(t)\\
&=&q^{-2}\sum_{t\in \mathbb{Z}_q^2}\chi(-t.\xi)q^2\sum_{\alpha  \in \mathbb{Z}_q^2}\widehat{E}(\alpha)\chi(t.\alpha)\widehat{E}(-\theta^{T}(\alpha))\\
&=&\sum_{\alpha \in \mathbb{Z}_q^2}\widehat{E}(\alpha)\widehat{E}(-\theta^{T}(\alpha))\sum_{t\in \mathbb{Z}_q^2} \chi(t.(\alpha-\xi))\\
&=&q^2\widehat{E}(\xi)\widehat{E}(-\theta^{T}(\xi)),
\end{eqnarray*}
we have
\begin{eqnarray*}
\sum_{\substack{\theta\in  SO_2(\mathbb{Z}_q)\\ \xi\in \mathbb{Z}_q^{2}\setminus (0,0)}}|\widehat{\nu}_{\theta}(\xi)|^2&=&q^4\sum_{\substack{\theta\in  SO_2(\mathbb{Z}_q)\\ \xi\in \mathbb{Z}_q^{2}\setminus (0,0)}}|\widehat{E}(\xi)|^2|\widehat{E}(-\theta^{T}(\xi))|^2\\
&=&q^4\sum_{\substack{\theta\in  SO_2(\mathbb{Z}_q)\\ \xi\in \mathbb{Z}_q^{2}\setminus (0,0)}}|\widehat{E}(\xi)|^2|\widehat{E}(\theta^{T}(\xi))|^2\\
&\leq&q^4\left(\max_{\xi \in \mathbb{Z}_q^{2}\setminus(0,0)}|Stab(\xi)|\right)\sum_{\xi\neq(0,0)}|\widehat{E}(\xi)|^2 \sum_{\substack{\eta \ne (0,0)\\ \| \eta \|=\| \xi \|} }   |\widehat{E}(\eta)|^2
\end{eqnarray*}

Plugging this value in (\ref{cube}) and using (\ref{L2}) we get 
\begin{eqnarray}\label{stabl2}
\sum_{(a,b,c)\in(\mathbb{Z}_q^2)^3/\sim}\mu^2(a,b,c)&\leq& \sum_{\substack{\theta\in  SO_2(\mathbb{Z}_q)\\t\in \mathbb{Z}_q^2}}\nu_{\theta}^3(t)\nonumber\\
&\lesssim&q^{-3}|E|^6+ 3q^6|E|\left(\max_{\xi \in \mathbb{Z}_q^{2}\setminus(0,0)}|Stab(\xi)|\right)\sum_{\xi\neq(0,0)}|\widehat{E}(\xi)|^2 \sum_{\substack{\eta \ne (0,0)\\ \| \eta \|=\| \xi \|} }   |\widehat{E}(\eta)|^2\nonumber\\
&=&q^{-3}|E|^6+ 3q^6|E|\mathrm{I}
\end{eqnarray}
where 
\begin{equation*}
\mathrm{I}=\left(\max_{\xi \in \mathbb{Z}_q^{2}\setminus(0,0)}|Stab(\xi)|\right)\sum_{\xi\neq(0,0)}|\widehat{E}(\xi)|^2 \sum_{\substack{\eta \ne (0,0)\\ \| \eta \|=\| \xi \|} }   |\widehat{E}(\eta)|^2
\end{equation*}
We first note that $|Stab(\xi)|\leq p^{l-1}$ for $ \xi  \ne (0,0)$  by Lemma \ref{nonzero norm} and \ref{zero norm}. Extending the summation in $\eta$ over all $\eta$ and using Plancherel twice in $\mathrm{I}$, we get
\begin{equation*}
 \mathrm{I}\leq p^{l-1}q^{-4}|E|^2.
\end{equation*}
Plugging this value in (\ref{stabl2}) gives 
\begin{equation}
\sum_{(a,b,c)\in(\mathbb{Z}_q^2)^3/\sim}\mu^2(a,b,c)\lesssim q^{-3}|E|^6+3q^2|E|^3p^{l-1},
\end{equation}
so that
\begin{eqnarray*}
|T_2(E)|&\geq& \frac{|E|^6}{q^{-3}|E|^6+3q^2|E|^3p^{l-1}}\\
&\geq&\frac{|E|^6}{2q^{-3}|E|^6}=\frac{q^3}{2}
\end{eqnarray*}
whenever $|E|\geq\sqrt[3]{3}p^{2l-\frac{1}{3}}$, which completes the proof.
\end{proof}

\vskip.125in

\section{Proof of Theorem \ref{Areainc}}
Now before giving the proof, let us introduce the necessary background.

Let $q=p^l$ and $(a,b)\in \mathbb{Z}_q^2$. Let $\left \langle(a,b)\right \rangle=\{t(a,b)\colon t\in \mathbb{Z}_q\}$ be the submodule of $\mathbb{Z}_q^2$ generated by $(a,b)$, which gives the line through origin and the point $(a,b)$ in $\mathbb{Z}_q^2$. Now, consider the set $$\Lambda_n=\{(a,b)\in \mathbb{Z}_q^2: p^n| a,b\; \text{but}\; (a,b)\neq(0,0)\; \text{mod}\;p^{n+1}\},$$
and denote $|\Lambda_n|=\lambda_n$ for $n=0,1, \cdots,l-1$.

\begin{lemma}
$\lambda_n=p^{2(l-n)}-p^{2(l-n-1)}$.
\end{lemma}

\begin{proof}
Since $p^n| a,b$ in $\mathbb{Z}_q$ we have $p^{l-n}$ choices for $a$ and $b$ each and hence $p^{2(l-n)}$ choices for $(a,b)$. Now we need to subtract $p^{2(l-n-1)}$ cases where $p^{n+1}$ divides both $a$ and $b$ to get the desired result. 
\end{proof}

\begin{lemma}  Let $\mathcal{L}_n=\{ \left \langle(a,b)\right \rangle :  (a,b)\in \Lambda_n \}$ denote the set of lines generated by the points of $\Lambda_n$. Then $|\mathcal{L}_n|=p^{l-n}+p^{l-n-1}.$
\end{lemma}
\begin{proof}
For $(a,b)\in \Lambda_n$, note that $\left \langle(a,b)\right \rangle$ is cyclic and $|\left \langle(a,b)\right \rangle|=p^{l-n}$. Hence there exist $\phi(p^{l-n})=p^{l-n}-p^{l-n-1}$ generators of the group which lies in $\Lambda_n$. So that for each n, we have 
$$\frac{p^{2(l-n)}-p^{2(l-n-1)}}{p^{l-n}-p^{l-n-1}}=p^{l-n}+p^{l-n-1}$$ many lines in  $\mathcal{L}_n$ each containing $p^{l-n}$ points.
\end{proof}

We now conclude that the average number of points in a line in $\mathbb{Z}_q^2$ is
\begin{eqnarray*}
&=&\frac{\sum_{n=0}^{l-1}(p^{l-n}+p^{l-n-1})p^{l-n}}{\sum_{n=0}^{l-1}p^{l-n}+p^{l-n-1}}\\
&=&\frac{p^{2l}+p^{2l-1}+p^{2l-2}+\cdots+p^2+p}{p^{l}+2p^{l-1}+2p^{l-2}+\cdots+2p+1}\\
&\sim& p^l
\end{eqnarray*}

In what follows we will only consider $\mathcal{L}_0$, i.e. the set of all lines of full length $q$ in $\mathbb{Z}_q^2$. 

\begin{lemma}\label{lines} For any  $(a,b)\in \mathbb{Z}_q^2$ if $(a,b)\in \Lambda_n$, then $(a,b)$ appears in $p^n$ distinct lines in $\mathcal{L}_0$.
\end{lemma}
\begin{proof} Say $(a,b)\in \Lambda_n$ and $p^{n+1}\nmid a$. Then $(a,b)$ belongs to the lines generated by $\left(\frac{a}{p^n}, ip^{l-n}+\frac{b}{p^n}\right)$ for $i=0,...,p^n-1$. Note that $i_0 p^{l-n}+\frac{b}{p^n}=i_1p^{l-n}+\frac{b}{p^n} \;\text{mod}\; p^l$ would imply 
\begin{eqnarray*}
i_0 p^{l-n}&=&i_1 p^{l-n} \;\text{mod}\; p^l\\
i_0&=&i_1\;\text{mod}\; p^n
\end{eqnarray*}
which is not the case. Hence the given points are all distinct. Since $\frac{a}{p^n}$ is a unit in $\mathbb{Z}_q$ it follows that the lines determined by the given generators are all distinct.
\end{proof}

\begin{lemma}\label{S}
Let $R_i=\{(x,y)\in \mathbb{Z}_q^2\times \mathbb{Z}_q^2: x-y\in\Lambda _i\}$  and $r_i=|R_i|$, for $i=1,...,l-1$. Then $r:=r_1p+r_2p^2+...+r_{l-1}p^{l-1}\leq 2p^{4l-1}$.
\end{lemma}
\begin{proof} Let $x=(x_1,x_2), y=(y_1,y_2)$ in $Z_q^2$. Now if $x-y=(x_1-y_1,x_2-y_2)\in \Lambda _i$ then $p^{i}|x_1-y_1$ and $x_2-y_2$ but $p^{i+1}\nmid x_1-y_1$ or $x_2-y_2$. $p^{i}|x_1-y_1$ gives $p^{l-i}$ choices for $x_1-y_1$ in $\mathbb{Z}_q$, we have $q$ choices for $y_1$ and $y_1$ determines $x_1$ uniquely. Hence  we have $qp^{l-i}$ choices for $x_1$ and $y_1$. Same argument applies for $x_2$ and $y_2$. Altogether the condition $p^{i}|x_1-y_1$ and $x_2-y_2$ gives $qp^{l-i}qp^{l-i}$ choices for $x=(x_1,x_2)$, $y=(y_1,y_2)$.

To exclude the cases where $p^{i+1}$ divides both $x_1-y_1$ and $x_2-y_2$ we need to subtract $qp^{l-(i+1)}qp^{l-(i+1)}$ cases of $x$ and $y$. Hence,

  $$r_i=qp^{l-i}qp^{l-i}-qp^{l-i-1}qp^{l-i-1}.$$ 

\noindent Now summing $r_ip^{i}$'s over all $i=1,...,l-1$ we get
\begin{eqnarray*}
r&=&(qp^{l-1}qp^{l-1}-qp^{l-2}qp^{l-2})p\\
&+&(qp^{l-2}qp^{l-2}-qp^{l-3}qp^{l-3} )p^2\\
&\vdots&\\
&+&(qpqp-q^2)p^{l-1}\\
&=&q^2(p^{2l-1}+p^{2l-2}-p^l-p^{l-1})\\
&\leq&2q^2p^{2l-1}=2p^{4l-1}
\end{eqnarray*}
\end{proof}


\begin{proof} [Proof of Theorem \ref{Areainc}]

Let $L$ be a line in $\mathcal{L}_0$ and consider the sum set 
$$E+L=\{e+l:e\in E,\: l\in L\}$$
Since $|L||E|>q^2$ 
$$e_1+l_1=e_2+l_2 \;\text{for some}\; l_1\neq l_2$$
so that 
\begin{eqnarray}\label{direction}
l_2-l_1=e_1-e_2.
\end{eqnarray}

Here we aim to average the solutions of the equation (\ref{direction}) over $p^l+p^{l-1}$ lines in $\mathcal{L}_0$. To start with, we count the number of solutions of (\ref{direction}) over all lines in $\mathcal{L}_0$ in two cases:

In the case $e_1=e_2$, there are $|E|$ and $q^2$ choices for  $e_1=e_2$ and $l_1=l_2$ , respectively. 

In the case $e_1\neq e_2$, we can choose $e_1$ and $e_2$ in $|E|(|E|-1)$ different ways, and once we fix them, we look at the difference $e_1-e_2$. At that point let $S_i=\{(e_1,e_2)\in E \times E: e_1-e_2\in \Lambda_i\}$ and $s_i=|S_i|$ for $i=0,...l-1$. We know from Lemma \ref{lines} that if $(e_1,e_2)\in S_i$, then  $e_1-e_2$ lies on $p^i$ lines in $\mathcal{L}_0$ and when we fix the line, $l_2-l_1$ can be written $q$ different ways on that line. In other words, for all $s_i$ pairs $(e_1,e_2)\in S_i$, $l_1,l_2$ is chosen $p^iq$ different ways over the lines in $\mathcal{L}_0$.

So altogether we have 
\begin{eqnarray*}
&&|\{(e_1,e_2,l_1,l_2)\in E\times E\times L\times L:\; (\ref{direction})\;\text{holds for some} L\in \mathcal{L}_0\}|\\
&=&|E|q^2+s_0q+s_1pq+s_2p^2q+...+s_{l-1}p^{l-1}q \\
&\leq&|E|^2q+(s_0+s_1p+s_2p^2+...+s_{l-1}p^{l-1})q\\
&=&|E|^2q+(|E|^2+s)q
\end{eqnarray*}
where $s=s_1p+s_2p^2+...+s_{l-1}p^{l-1}$. Note that $s\leq r\leq 2p^{4l-1}\leq2|E|^2$ by Lemma \ref{S} and the assumption on the size of $E$. 

Hence we get,
$$|\{(e_1,e_2,l_1,l_2)\in E\times E\times L\times L:\; (\ref{direction})\;\text{holds for some} L\in \mathcal{L}_0\}|\leq 4|E|^2q$$

\noindent It follows that there exists a $L\in \mathcal{L}_0$ such that  
\begin{eqnarray*}
|\{(e_1,e_2,l_1,l_2)\in E\times E\times L\times L :\; (\ref{direction})\;\text{holds}\}| &\leq& 4|E|^2\frac{p^{l}}{p^l+p^{l-1}}\\
&=&4|E|^2\frac{p}{1+p}.
\end{eqnarray*}

If we let $\nu_{E+L}(n)$ denote the number of representation of $n$ as $e+l$ for some  $e\in E$, $l\in L$, then by Cauchy-Schwarz, for this particular $L$,
\begin{eqnarray*}
|E|^2|L|^2&=&\left(\sum_{n\in E+L}\nu(n)\right)^2\\
&\leq&|E+L|\sum_{n\in E+L}\nu^{2}(n)\\
&=&|E+L||\{(e_1,e_2,l_1,l_2)\in E\times E\times L\times L:\; (\ref{direction})\;\text{holds}\}|.
\end{eqnarray*}

Hence,
\begin{eqnarray*}
|E+L|&\geq& \frac{|E|^2|L|^2}{|\{(e_1,e_2,l_1,l_2)\in E\times E\times L\times L:\; (\ref{direction})\;\text{holds}\}|}\\
&\geq& \frac{|E|^2p^{2l}}{4|E|^2\frac{p}{1+p}}=\frac{q^2}{4}\frac{1+p}{p}.
\end{eqnarray*}

We conclude that there exist points of $E$ in at least $\frac{q}{4}\frac{1+p}{p}$ parallel lines. Here we shall note that there are totally $q$ parallel lines to $L$, including itself, and one of them must contain two points of $E$ with a unit distance in between. For otherwise, on each of these parallel lines there would be at most $p^{l-1}$ points of $E$ yielding $|E|\leq qp^{l-1}=p^{2l-1}$ which is not the case. It follows that points of $E$ determines at least $\frac{q}{4}\frac{1+p}{p}-1$ distinct triangle areas with the same unit base on one of the parallel lines and with the third vertex being on  $\frac{q}{4}\frac{1+p}{p}-1$ different parallel lines.

\end{proof}

\vskip.125in 

\section{Proof of Theorem \ref{dot E} } 
\begin{proof}
Let $$\nu(t)=\{(x,y)\in E\times E:x.y=t\}.$$
Then by the Cauchy-Schwarz inequality,
\begin{equation}\label{cs prod}
|E|^4=\left(\sum_{t\in \mathbb{Z}_q}\nu(t)\right)^2\leq |\prod(E)|\sum_{t\in \mathbb{Z}_q}\nu(t)^2
\end{equation}
We can write
$$\nu(t)=\sum_{x\in E}\sum_{x.y=t}E(y).$$
From the Cauchy-Schwarz inequality, it follows that
 $$\nu^2(t)\leq|E|\sum_{x\in E}\sum_{x.y=x.y'=t}E(y)E(y')$$
 so that 
\begin{eqnarray}\label{l2 ei}
\sum_{t}\nu^2(t)&\leq &|E|\sum_{x.y=x.y'}E(x)E(y)E(y')\nonumber\\
&=&|E|q^{-1}\sum_{s}\sum_{x,y,y'}\chi(s(x.y-x.y'))E(x)E(y)E(y')\nonumber\\
&=&\frac{|E|^4}{q}+\sum_{i=0}^{l-1}e_i,
\end{eqnarray}
where
\begin{eqnarray*}
e_{i}&=&|E|q^{-1}\sum_{\substack{s\\ \upsilon_p(s)=i}}\sum_{x,y,y'}\chi(sx(y-y'))E(x)E(y)E(y')\\
&=&|E|q^{2d-1}\sum_{\substack{s\\ \upsilon_p(s)=i}}\sum_{x}E(x)|\widehat{E}(sx)|^2\\
&=&|E|q^{2d-1}\sum_{s'\in\mathbb{Z}_{p^{l-i}}^*}\sum_{x}E(x)|\widehat{E}(p^{i}s'x)|^2\\
&=&|E|q^{2d-1}\sum_{s'\in\mathbb{Z}_{p^{l-i}}^*}\sum_{x}E(s'x)|\widehat{E}(p^{i}x)|^2.
\end{eqnarray*}
Now denoting $l_x^i=\{s'x:s'\in \mathbb{Z}_{p^{l-i}}^*\}$, we can write
$$e_i=|E|q^{2d-1}\sum_{x}|E\cap l_x^i||\widehat{E}(p^{i}x)|^2.$$

We note here that $|E\cap l_x^i|\leq |A|=p^{\alpha}$. To see this, let $x=(x_1,\ldots, x_d)$ and 
$$\overline{l_x^i}:=\{sx: s\in \mathbb{Z}_{p^{l-i}}\}\supset\{s'x:s'\in \mathbb{Z}_{p^{l-i}}^*\}= l_x^i$$
From the definition, it is clear that $|E\cap l_{x}^i|\leq|E\cap \overline{l_{x}^{i}}|$ and we will show that $|E\cap\overline{l_x^i}|\leq|A|$ where $E=\underbrace{A\times \ldots \times A}_\text{\text{d-fold}}$.

Note that we can assume that at least one of the coordinates of $x=(x_1,\ldots,x_d)$ is unit. Since otherwise if $\upsilon_{p}(x_j)=n>0$ is among the smallest, we can write $x=p^n\tilde{x}$ where $\tilde{x}=(\tilde{x_1},\ldots,\tilde{x_j},\ldots,\tilde{x_d})$ with $\tilde{x_j}$ is unit and $\overline{l_{\tilde{x}}^i}=\overline{l_{x}^{i}}$ gives $|E\cap \overline{l_{\tilde{x}}^i}|=|E\cap\overline{l_{x}^{i}} |$.

So we assume that  $x_j$ is unit. If $\{s_1x,\ldots,s_tx\}=E\cap \overline{l_{x}^{i}}$, then the $j$th coordinates $s_kx_j$ of the vectors $s_kx$, $1\leq k\leq t$, are all different in $A$. Hence $t\leq|A|.$ Thus,
$$|E\cap l_{x}^i|\leq|E\cap \overline{l_{x}^{i}}|=t\leq |A|.$$

Clearly, we also have $|E\cap l_x^i|\leq p^{l-i}$. Therefore $|E\cap l_x^i|\leq min\{p^{l-i},|A|\}$, and 
\begin{eqnarray}\label{mid} 
e_i&\leq&|E|q^{2d-1}\sum_{x}min\{p^{l-i},p^{\alpha}\}|\widehat{E}(p^{i}x)|^2.
\end{eqnarray}
Now we estimate 
\begin{eqnarray*}
\sum_{x}|\widehat{E}(p^{i}x)|^2&=&q^{-2d}\sum_{x}\sum_{u,v}\chi((u-v)p^{i}x)E(u)E(v)\\
&=&q^{-d}\sum_{(u-v)p^{i}=0}E(u)E(v)
\end{eqnarray*}
For $u=(u_1,\ldots,u_d), v=(v_1,\ldots,v_d) \in \underbrace{A\times \ldots \times A}_\text{\text{d-fold}}$, if $$p^{i}(u-v)=p^i(u_1-v_1,\ldots,u_d-v_d)=(0,\ldots,0),$$ then $p^{i}(u_j-v_j)=0$ for all $1\leq j \leq d$. That implies $u_j-v_j\in\{p^{l-i}, 2p^{l-i},\ldots, p^{i}p^{l-i}\}$ and we have $p^i$ choices for $u_j-v_j$. If we first fix $v_j$ in $|A|$ different ways, then $u_j$ is uniquely determined. So we have $|A|p^i$ choices for each $(u_j,v_j)$, giving $|A|^dp^{id}$ choices for $(u,v)$.

It follows that
\begin{eqnarray*}
\sum_{x}|\widehat{E}(p^{i}x)|^2&\leq&q^{-d}p^{id}|A|^d\\
&=&|E|q^{-d}p^{id}
\end{eqnarray*}
plugging this value in (\ref{mid}) and summing over all $e_i's$ for $i=0,...,l-1$, we get
\begin{eqnarray*}
\sum_{i=0}^{l-1}e_i&\leq&|E|^2q^{d-1}\sum_{i=0}^{l-1}min\{p^{l-i},p^{\alpha}\}p^{id}\\
&=&|E|^2q^{d-1}\left(\sum_{i=0}^{l-\alpha}p^{\alpha}p^{id}+\sum_{i=l-\alpha+1}^{l-1}p^{l-i}p^{id}\right)\\
&\lesssim&|E|^2q^{d-1}(p^{\alpha}p^{(l-\alpha)d}+pp^{(l-1)d})\\
&=&|E|^2q^{d-1}p^{ld}(p^{\alpha(1-d)}+p^{1-d})\\
&\lesssim&|E|^2q^{d-1}p^{ld}p^{1-d}\\
&=&|E|^2q^{2d-1}q^{\frac{1-d}{l}}.
\end{eqnarray*}
Plugging this value in (\ref{l2 ei}), the inequality (\ref{cs prod}) yields
$|\prod(E)|\gtrsim\frac{|E|^4}{|E|^2q^{2d-1}q^{\frac{1-d}{l}}}\gtrsim q$, whenever $|E|^2\gtrsim q^{2d+\frac{1-d}{l}}$ i.e. $|E|\gtrsim q^{d(\frac{2l-1}{2l})+\frac{1}{2l}}$.
\end{proof}

\end{document}